\documentclass[11pt,leqno]{amsart}
\usepackage{amsmath,amssymb,amsthm}
\usepackage{hyperref}
\usepackage[UKenglish]{babel}
\usepackage{microtype}
\usepackage{color}
\usepackage[shortlabels]{enumitem}

\usepackage{xcolor}
\hypersetup{
    colorlinks,
    citecolor={blue},
}

\newtheorem{theorem}{Theorem}[section]
\newtheorem{lemma}[theorem]{Lemma}

\newtheorem{proposition}[theorem]{Proposition}
\newtheorem{corollary}[theorem]{Corollary}

\theoremstyle{definition}
\newtheorem{definition}[theorem]{Definition}

\newtheorem{example}[theorem]{Example}

\theoremstyle{remark}
\newtheorem{remark}[theorem]{Remark}
\numberwithin{equation}{section}

\def\fnote#1{\footnote}

\def\ignora#1{}
\def\n3#1{\left\vert  \! \left\vert \! \left\vert \, #1 \, \right\vert \!
  \right\vert \! \right\vert }

\newcommand{\pten}{\ensuremath{\widehat{\otimes}_\pi}}
\def\<{\langle}
\def\>{\rangle}

\def\co{\operatorname{co}}

\newcommand{\conv}{\mathop\mathrm{conv}}

\newcommand{\ext}[1]{\mathrm{ext}\left(#1\right)}


\author{Abraham Rueda Zoca}\thanks{The research of the author was supported by MECD (Spain) FPU2016/00015, by MINECO (Spain) Grant MTM2015-65020-P and by Junta de Andaluc\'ia Grants FQM-0185.}
\address{Universidad de Granada, Facultad de Ciencias.
Departamento de An\'{a}lisis Matem\'{a}tico, 18071-Granada
(Spain)} \email{ abrahamrueda@ugr.es}
\urladdr{\url{https://arzenglish.wordpress.com}}


\title{Almost squareness and strong diameter two property in tensor product spaces}

\begin{document}

\maketitle
\markboth{ABRAHAM RUEDA ZOCA}{ALMOST SQUARENESS AND SSD2P IN TENSOR PRODUCT SPACES}

\begin{abstract}
We study almost squareness and the strong diameter two property in the setting of projective (symmetric) tensor product. We prove that almost squareness is preserved by taking projective tensor product, providing non-trivial examples of ASQ projective tensor product spaces. Furthermore, we give sufficient conditions for a projective symmetric tensor product to have the strong diameter two property which extend most of the known results and provide new examples of such spaces with the strong diameter two property.
\end{abstract}

\section{Introduction}

Let $X$ be a (real) Banach space
with closed unit ball $B_X$ and unit sphere $S_X$. We will say that
\begin{enumerate}
\item
  $X$ has the \emph{slice diameter two property} (slice-D2P) whenever
  each slice of $B_X$ has diameter two.
\item
  $X$ has the \emph{diameter two property} (D2P) whenever each
  non-empty relatively weakly open subset of $B_X$ has diameter
  two.
\item
  $X$ has the \emph{strong diameter two property} (SD2P)
  whenever each convex combination of slices of $B_X$ has
  diameter two.
\end{enumerate}
Similarly, for dual spaces one can define the $w^*$-LD2P, the $w^*$-D2P and the $w^*$-SD2P by replacing slices and weakly open subsets with weak$^*$ slices and weak$^*$ open subsets in the above definitions. It is known that all the above properties are different in an extreme way \cite{blradv}.

In \cite[Section 5, (b)]{aln} it is wondered how are the diameter two properties, in general, preserved by taking tensor products spaces. In the case of the projective tensor product, it is straightforward to check that the slice-D2P is preserved by one factor by taking projective tensor product \cite[Theorem 2.7, (i)]{aln}. The question for the SD2P is more subtle, for which it was proved in \cite[Corollary 3.6]{blr} that the SD2P is preserved from both factors by taking projective tensor product but not from just one of them \cite[Corollary 3.9]{llr2}. See \cite{hlp,llr} for further results in this line.

This nice behaviour of the diameter two properties in tensor product spaces motivated stability results of stronger properties, coming from \cite{all}, by tensor product spaces in \cite{llr}. According to \cite{all}, a Banach space $X$ is
\begin{enumerate}
\item
  \emph{locally almost square} (LASQ) if for every $x\in S_X$
  there exists a sequence $\{y_n\}$ in $B_X$ such that
  $\Vert x\pm y_n\Vert\rightarrow 1$ and $\Vert y_n\Vert\rightarrow 1$.
\item
  \emph{weakly almost square} (WASQ) if for every $x\in S_X$
  there exists a sequence $\{y_n\}$ in $B_X$ such that
  $\Vert x\pm y_n\Vert\rightarrow 1$, $\Vert y_n\Vert\rightarrow 1$
  and $y_n \rightarrow 0$ weakly.
\item
  \emph{almost square} (ASQ) if for every $x_1,\ldots, x_k \in S_X$
  there exists a sequence $\{y_n\}$ in $B_X$ such
  that $\Vert y_n\Vert\rightarrow 1$ and $\Vert x_i\pm y_n\Vert\rightarrow 1$
  for every $i\in\{1,\ldots, k\}$.
\end{enumerate}

It is known that the sequence involved in the definition of ASQ can be chosen to be weakly-null \cite[Theorem~2.8]{all}, so
ASQ implies WASQ which in turn implies LASQ. Moreover, ASQ implies the SD2P,
WASQ implies the D2P, and LASQ implies the LD2P, see \cite{all}.

In \cite[Theorem 2.6]{llr} it is proved that almost squareness preserved by one factor by taking injective tensor product. In view of \cite[Corollary 3.6]{blrtenso}, it is a natural question whether almost squareness is preserved from both factors by taking projective tensor product. However, to the best of our knonwledge, it is even not known any non-trivial example of ASQ projective tensor product space (see \cite[Remark 2.12]{llr}). We will prove in Theorem~\ref{teonosimetri} that almost squareness is preserved from both factor by taking projective tensor product. This stability result will be proved by making use of classical Rademacher techniques \cite[Section 2.5]{rya} employed for studying tensor diagonal bases.

In contrast with the projective tensor product, no stability result of diameter two properties is known for the projective symmetric tensor products. It is known that all the projective symmetric tensor products of a Banach space $X$ has the D2P whenever $X$ has an infinite-dimensional centralizer and $B_X$ contain any extreme point \cite[Corollary 2.6]{ab}. Furthermore, in \cite[Theorem 3.3]{blr} it is proved that the all the projective symmetric tensor products of an ASQ Banach space enjoy the SD2P. In both cases, it is extremely important having a $c_0$ behaviour in which $X$ has the SD2P because this allows us to construct weakly null sequences in the projective symmetric tensor product (see \cite[Lemma 3.1]{blr}). Bearing this fact in mind we will introduce the \textit{sequential strong diameter two property (sequential SD2P)} in Definition \ref{defiseqsd2p}, a stregthening of the SD2P  motivated by the \textit{symmetric strong diameter two property} (see definition below) introduced in \cite{anp}, proving that if a Banach space $X$ has such property then every projective symmetric tensor product of $X$ has the SD2P in Theorem~\ref{teosymmetricstrong}. In spite of the fact that the sequential SD2P seems to be quite technical property to check in a Banach space, we will prove that it can be applied to several Banach spaces which are well known to have the SD2P as are infinite-dimensional uniform algebras or Banach spaces with an infinite-dimensional centralizer and whose unit ball contains any extreme point, and also the space $\mathrm{Lip}_0(M)$ whenever $M$ is a metric space with infinitely-many cluster points. As a consequence of the previous examples, we will generalise the main results of \cite[Section 3]{blr} and \cite{ab} as well as we will produce new examples of projective symmetric tensor product spaces enjoying the SD2P. 

\textbf{Notation:} Throughout the paper we will only consider real Banach spaces. Given a Banach space $X$, by a \textit{slice of $B_X$} we will mean a set of the form
$$S(B_X,f,\alpha):=\{x\in B_X: f(x)>1-\varepsilon\},$$
where $f\in S_{X^*}$ and $\alpha>0$. If $X$ is a dual Banach space, the previous set will be a \textit{weak-star slice} if $f$ belongs to the predual of $X$. For else standard notation about Banach spaces we refer to \cite{alka}. According to \cite{anp} a Banach space $X$ is said to have the \textit{symmetric strong diameter two property} (SSD2P) if, for every $n\in\mathbb N$, slices $S_1,\ldots, S_n$ of $B_X$ and $\varepsilon>0$ there are $x_i\in S_i$ for every $i\in\{1,\ldots, n\}$ and $y\in B_X$ such that $x_i\pm y\in S_i$ holds for every $i\in\{1,\ldots, n\}$ and $\Vert y\Vert>1-\varepsilon$. It is clear that ASQ Banach spaces enjoy the SSD2P and that the SSD2P implies the SD2P. It is known that the converse of the above implications do not hold \cite{hlln}.

Given two Banach spaces $X$ and $Y$ we will denote by $L(X,Y)$ the space of all bounded linear operators from $X$ to $Y$. Recall that the \textit{projective tensor product} of $X$ and $Y$, denoted by $X\pten Y$, is the completion of $X\otimes Y$ under the norm given by
\begin{equation*}
   \Vert u \Vert :=
   \inf\left\{
      \sum_{i=1}^n  \Vert x_i\Vert\Vert y_i\Vert
      : u=\sum_{i=1}^n x_i\otimes y_i
      \right\}.
\end{equation*}
It is known that $B_{X\pten Y}=\overline{\conv}(B_X\otimes B_Y)
=\overline{\conv}(S_X\otimes S_Y)$ \cite[Proposition~2.2]{rya}.
Moreover, given Banach spaces $X$ and $Y$, it is well known that
$(X\pten Y)^*=L(X,Y^*)$ (see \cite{rya} for background).

We pass now to projective symmetric tensor products.
Given a Banach space $X$, we define the
\textit{($N$-fold) projective symmetric tensor product} of $X$, denoted by
$\widehat{\otimes}_{\pi,s,N} X$, as the completion of the space
$\otimes^{s,N}X$ under the norm
\begin{equation*}
   \Vert u\Vert:=\inf
   \left\{
      \sum_{i=1}^n \vert \lambda_i\vert \Vert x_i\Vert^N :
      u:=\sum_{i=1}^n \lambda_i x_i^N, n\in\mathbb N, x_i\in X
   \right\}.
\end{equation*}
The dual, $(\widehat{\otimes}_{\pi,s,N} X)^*=\mathcal P(^N X)$, is
the Banach space of $N$-homogeneous continuous polynomials on $X$, and notice that $B_{\widehat{\otimes}_{\pi,s,N} X}=\overline{\co}(\{x^N:x\in S_X\})$ (see \cite{flo} for background). We will denote by $P(^N X,Y)$ the space of all $N$-homogeneous polynomials from $X$ to $Y$.

Given a metric space $M$ we will denote by $B(x,r):=\{y\in M: d(x,y)<r\}$, where $x\in M$ and $r>0$.

\section{Main results}\label{section:main}

We will begin this section with a study of how almost squareness is preserved by taking projective tensor product. This study is motivated by the fact that the SD2P is preserved by taking projective tensor product from both factors \cite[Corollary 3.6]{blrtenso} and by \cite[Remark 2.12]{llr}, where it is asked whether there are two Banach spaces $X$ and $Y$ such that $X\pten Y$ is ASQ and such that $\min\{\dim(X),\dim(Y)\}\geq 2$. The following theorem gives a large class of such non-trivial examples.

\begin{theorem}\label{teonosimetri}
Let $X$ and $Y$ be Banach spaces. If $X$ and $Y$ are ASQ then so is $X\pten Y$.
\end{theorem}

For the proof we will need the following lemma.

\begin{lemma}\label{lematecniasq}
Let $X$ and $Y$ be two Banach spaces. Pick $x,y\in S_X, u,v\in S_Y$ such that
$$\Vert x\pm y\Vert\leq 1+\varepsilon, \Vert u\pm v\Vert\leq 1+\varepsilon$$
hold. Then
$$\Vert x\otimes u\pm y\otimes v\Vert\leq (1+\varepsilon)^2.$$
\end{lemma}

\begin{proof}
Pick $x,y\in S_X, u,v\in S_Y$ as in the hypotheses of the lemma. Then
\[\begin{split}
x\otimes u+y\otimes v& =\frac{1}{4}((x+y)\otimes (u+v)\\
& +(-x+y)\otimes (-u+v)\\
& +(x-y)\otimes (u-v)\\ & +(-x-y)\otimes (-u-v)),
\end{split}
\]
from where the triangle inequality implies
$$\Vert x\otimes u+y\otimes v\Vert \leq \frac{1}{4}(4 (1+\varepsilon)^2)=(1+\varepsilon)^2,$$
as desired. The other inequality follows in a similar way.
\end{proof}

\begin{proof}[Proof of Theorem \ref{teonosimetri}]
Let $z_1,\ldots z_n\in S_{X\pten Y}$ and $\varepsilon>0$, and let us find $z\in S_{X\pten Y}$ such that $\Vert z_i+z\Vert < 1+\varepsilon$ holds for every $i\in\{1,\ldots, n\}$. To this aim, since $B_{X\pten Y}=\overline{\co}(S_X\otimes S_Y)$, then we can find, for every $i\in\{1,\ldots, n\}$, elements of $X\pten Y$ such that
\begin{equation}\label{aproximatensor}
\left\Vert z_i-\sum_{j=1}^{n_i} \lambda_{ij}x_{ij}\otimes y_{ij}\right\Vert<\frac{\varepsilon}{2},
\end{equation}
where $n_i\in\mathbb N$, $x_{ij}\in S_X, y_{ij}\in S_{Y}$ holds for every $j\in\{1,\ldots, n_i\}$ and $\lambda_{i1},\ldots, \lambda_{in_i}\in [0,1]$ satisfy that $\sum_{j=1}^{n_i}\lambda_{ij}=1$. Now, since $X$ is ASQ we can find $x\in S_X$ such that
\begin{equation}\label{nosimortox}
\Vert x_{ij}\pm x\Vert<\left(1+\frac{\varepsilon}{2}\right)^\frac{1}{2}
\end{equation}
holds for every $i\in\{1,\ldots, n\}$ and $j\in\{1,\ldots, n_i\}$. Furthermore, since $Y$ is ASQ, we can find $y\in S_Y$ such that
\begin{equation}\label{nosimortoy}
\Vert y_{ij}\pm y\Vert<\left(1+\frac{\varepsilon}{2}\right)^\frac{1}{2}
\end{equation}
holds for every $i\in\{1,\ldots, n\}$ and $j\in\{1,\ldots, n_i\}$. Define $z:=x\otimes y\in S_X\otimes S_Y$. Given $i\in\{1,\ldots, n\}$ and $j\in\{1,\ldots, n_i\}$ then Lemma \ref{lematecniasq} implies that
$$\Vert x_{ij}\otimes y_{ij}+x\otimes y\Vert<1+\frac{\varepsilon}{2}$$
by using \eqref{nosimortox} and \eqref{nosimortoy}. Consequently, given $i\in\{1,\ldots, n\}$, we have
\[
\begin{split}
\Vert z_i+z\Vert& \mathop{<}\limits^{\mbox{\eqref{aproximatensor}}}\frac{\varepsilon}{2}+\left\Vert \sum_{j=1}^{n_i}\lambda_{ij} x_{ij}\otimes y_{ij} +x\otimes y\right\Vert\\
 & =\frac{\varepsilon}{2}+\left\Vert \sum_{j=1}^{n_i}\lambda_{ij} (x_{ij}\otimes y_{ij} +x\otimes y)\right\Vert\\
 & \leq \frac{\varepsilon}{2}+\sum_{j=1}^n \lambda_{ij}\Vert x_{ij}\otimes y_{ij}+x\otimes y\Vert\\
 & <\frac{\varepsilon}{2}+\sum_{j=1}^{n_i} \lambda_{ij}\left(1+\frac{\varepsilon}{2}\right)=1+\varepsilon.
\end{split}
\]
Since $i\in\{1,\ldots, n\}$ was arbitrary we conclude the desired result.
\end{proof}

\begin{remark}\label{consenosimetri}
In \cite[Remark 2.12]{llr} it is asked whether there is any non-trivial example of an ASQ projective tensor product space. Theorem \ref{teonosimetri} goes further and even provides a stability result of such property.
\end{remark}

\begin{remark}
In \cite[Question 4.3]{blr} it is asked whether almost squareness is preserved by taking projective symmetric tensor product. To the best of our knownledge, this question is still open. Notice that, in order to give a positive answer, an extension of Lemma \ref{lematecniasq} to the projective symmetric tensor product is needed. However, it is not obvious that such result extends to the symmetric case except to the case of two factors. Indeed, given a Banach space $X$, consider $Y:=\widehat{\otimes}_{\pi,s,2} X$. Then, given $x, y\in S_X$, it follows from classical addition formula \cite[pp. 8]{flo} that the following equality
$$x^2+y^2=\frac{(x+y)^2+(x-y)^2}{2},$$
holds in $Y$, from where similar techniques to the ones of the proof of Theorem \ref{teonosimetri} imply the following corollary.
\end{remark}
\begin{corollary}
Let $X$ be an ASQ Banach space. Then $\widehat{\otimes}_{\pi,s, 2} X$ is ASQ.
\end{corollary}

Now we turn to find a sufficient condition for a projective symmetric tensor product to have the SD2P. To this end, we need to consider a stregthening of the SSD2P. Notice that in \cite[Theorem 2.1]{hlln} it is proved that a Banach space $X$ has the SSD2P if, and only if, for every $n\in\mathbb N, x_1,\ldots, x_n\in S_X$ we can find nets $\{y_\alpha^i\},\{z_\alpha\}$ in $S_X$ such that $\{y_\alpha^i\}\rightarrow x_i$ and $\{z_\alpha\}\rightarrow 0$ weakly, and $\Vert y_\alpha^i\pm z_\alpha\Vert\rightarrow 1$.

Let us now consider the following definition.
\begin{definition}\label{defiseqsd2p}
Let $X$ be a Banach space. We say that $X$ has the \textit{sequential strong diameter two property (sequential SD2P)} if, for every $N\in\mathbb N$ and every $x_1,\ldots, x_N\in S_X$, there are sequences $\{y_n^i\},\{z_n\}$ of $S_X$ for every $i\in\{1,\ldots, N\}$ such that $\{x_i-y_n^i\}$ is either norm-null or equivalent to the $c_0$ basis for every $i$, $\{z_n\}$ is equivalent to the $c_0$ basis and $\Vert y_n^i\pm z_n\Vert\rightarrow 1$ holds for every $i\in\{1,\ldots, N\}$.
\end{definition}

The definition of the sequential SD2P encondes, roughly speaking, a $c_0$ way in which a Banach space $X$ has the SD2P. The reason why we look for such beaviour is that, given a Banach space $X$, if a sequence $\{x_n\}$ is equivalent to the $c_0$ basis then, for every $N\in\mathbb N$ and $P\in \mathcal P(^N X)$, it follows that $P(x_n)\rightarrow 0$, so we get that $\{x_n^N\}$ is weakly null in $\widehat{\otimes}_{\pi,s, N} X$ (see \cite[Lemma 3.1]{blr} for details). Notice that this has been the key idea to proving \cite[Proposition 2.4]{ab} as well as \cite[Theorem 3.3]{blr}.

On the other hand, note that sequential SD2P is formally stronger than the SSD2P, but we do not know whether they are equivalent. It is clear that sequential SD2P implies the containment of $c_0$, but it is not known whether all the Banach spaces with the SSD2P contain an isomorphic copy of $c_0$ \cite[Question 6.1]{hlln}. We will give examples of Banach spaces with the sequential SD2P in Section \ref{section:examples}

Our interest in the sequential SD2P comes from the following theorem.

\begin{theorem}\label{teosymmetricstrong}
Let $X$ be a Banach space and let $N\in\mathbb N$. If $X$ has the sequential SD2P then, for every $P_1,\ldots, P_n\in S_{\mathcal P(^N X)}$ and every $\varepsilon>0$ there exists $f\in S_{X^*}$ such that, if we define $Q:=f^N\in \mathcal P(^N X)$, we have that
$$\Vert P_i+Q\Vert>2-\varepsilon$$
holds for every $i\in\{1,\ldots, n\}$.
\end{theorem}

\begin{proof}
Pick $P_1,\ldots, P_n\in S_{\mathcal P(^N X)}$ and $\varepsilon>0$. Since $\Vert P_i\Vert=1$ then we can find $x_i\in S_X$ such that $P_i(x_i)>1-\varepsilon$ for every $i\in\{1,\ldots, n\}$. By assumptions we can find sequences $\{y_k^i\}$ and $\{z_k\}$ in the unit sphere such $\{x_i-y_k^i\}$ is either norm-null or equivalent to the $c_0$ basis for every $i$, $\{z_k\}$ is equivalent to the $c_0$ basis and $\Vert y_k^i\pm z_k\Vert\rightarrow 1$ holds for every $i\in\{1,\ldots, N\}$. Now pick an arbitrary $i\in\{1,\ldots, n\}$. Notice that if $\{x_i-y_k^i\}$ is norm null then $P_i(x_i-y_k^i)\rightarrow 0$ because of the continuity of $P_i$. On the other hand, if $\{x_i-y_n^i\}$ is equivalent to the $c_0$ basis then $P_i(x_i-y_k^i)\rightarrow 0$ by \cite[Lemma 3.1]{blr}. Consequently, by \cite[Lemma 1.1]{fajo} we get that $P_i(y_k^i)\rightarrow P_i(x_i)$. Since $\{z_k\}$ is equivalent to the $c_0$ basis again an application of \cite[Lemma 3.1]{blr} and \cite[Lemma 1.1]{fajo} implies that $P_i(y_k^i+z_k)\rightarrow P_i(x_i)>1-\varepsilon$. Taking into account that $\Vert y_k^i\pm z_k\Vert\rightarrow 1$ we can find $k$ large enough so that $\Vert y_k^i\pm z_k\Vert\leq 1+\varepsilon$ and $P_i(y_k^i+z_k)>1-\varepsilon$ hold. Pick $f\in S_{X^*}$ such that $f(z_k)=1$ and define $Q:=f^N\in S_{\mathcal P(^N X)}$. In order to finish the proof let us prove that
$$\Vert P_i+Q\Vert>\frac{1-\varepsilon+(1-\varepsilon)^N}{(1+\varepsilon)^N},$$
which implies the thesis of the theorem from the arbitrariness of $\varepsilon$ and since the definition of $Q$ does not depend on $i$. To this end, notice that
$$1+\varepsilon>\Vert y_k^i\pm z_k\Vert\geq \vert f(y_k^i)\pm f(z_k)\vert.$$
Now, if we make a correct choice of sign, we conclude that
$$1+\varepsilon>\vert f(y_k^i)\vert+\vert f(z_k)\vert=1+\vert f(y_k^i)\vert,$$
from where we obtain that $\vert f(y_k^i)\vert<\varepsilon$. Taking into account that $\Vert y_k^i\pm z_k\Vert<1+\varepsilon$ we get
\[
\begin{split}
\Vert P_i+Q\Vert>\frac{P_i(y_k^i+z_k)+Q(y_k^i+z_k)}{\Vert y_k^i+z_k\Vert^N}& >\frac{1-\varepsilon+(f(z_k)+f(y_k^i))^N}{\Vert y_k^i+z_k\Vert^N}\\
 & >\frac{1-\varepsilon+(1-\varepsilon)^N}{\Vert y_k^i+z_k\Vert^N}\\
 & >\frac{1-\varepsilon+(1-\varepsilon)^N}{(1+\varepsilon)^N},
\end{split}
\]
as desired.
\end{proof}

Recall that the norm of a Banach space $X$ is said to be \textit{octahedral} if, for every finite-dimensional subspace $Y$ of $X$ and every $\varepsilon>0$, there exists $x\in S_X$ such that
$$\Vert y+\lambda x\Vert\geq (1-\varepsilon)(\Vert y\Vert+\vert\lambda\vert)$$
holds for every $y\in Y$ and every $\lambda\in\mathbb R$. From an isomorphic point of view, it is known that a Banach space $X$ admits an equivalent octahedral norm if, and only if, $X$ contains an isomorphic copy of $\ell_1$ \cite{god}. From an isometric point of view, the norm of a Banach space $X$ is octahedral if, and only if, $X^*$ has the $w^*$-SD2P \cite[Theorem 2.1]{blrocta}. As a consequence of a weak-star density argument, a Banach space $X$ has the SD2P if, and only if, the norm of $X^*$ is octahedral \cite[Corollary 2.2]{blrocta}. Notice that Theorem \ref{teosymmetricstrong} implies that if a Banach space $X$ has the sequential SD2P then the norm of $\mathcal P(^N X)$ is octahedral. Bearing the above facts in mind we have the following corollary.

\begin{corollary}\label{coromejoraSD2P}
Let $X$ be a Banach space with the sequential SD2P and let $N\in\mathbb N$. Then:
\begin{enumerate}
\item\label{coro1} The norm of $\mathcal P(^NX)$ is octahedral. Equivalently, $\widehat{\otimes}_{\pi, s, N} X$ has the SD2P.
\item\label{coro2} If $Y$ is a Banach space with an octahedral norm, then the norm of $\mathcal P(^NX,Y)$ is octahedral.
\end{enumerate}
\end{corollary}

\begin{proof}
The first assertion is a direct consequence of Theorem~\ref{teosymmetricstrong}, \cite[Proposition 2.1]{hlp} and \cite[Corollary 2.2]{blrocta}. To get the second assertion, notice that $\mathcal P(^NX,Y)$ is isometrically isomorphic to $L(\widehat{\otimes}_{\pi,s,N}X,Y)$ \cite{flo}. Now, since the norms of $(\widehat{\otimes}_{\pi,s,N}X)^*=\mathcal P(^N X)$ and $Y$ are octahedral, so is the norm of $L(\widehat{\otimes}_{\pi,s,N}X,Y)$ by \cite[Theorem 3.5]{blrtenso}, and we are done.
\end{proof}

We will end this section giving an extension of Theorem \ref{teosymmetricstrong} which will allow us to improve the main result of \cite{ab}.  In order to do so, we need to introduce a bit of notation. Given a Banach space $X$, a natural number $N\in\mathbb N$ and a polynomial $P\in \mathcal P(^N X)$, then $P$ can be extended in a canonial way (the so-called Aron-Berner extension) to a polynomial $\widehat{P}\in \mathcal P(^N X^{**})$ such that $\Vert\widehat{P}\Vert=\Vert P\Vert$ (see \cite{ab,dg}). Taking into account such extension we can improve Theorem \ref{teosymmetricstrong} in the following sense.

\begin{proposition}\label{propobiduSD2P}
Let $X$ be a Banach space such that $X^{**}$ has the sequential SD2P and $N\in\mathbb N$. Then for every $P_1,\ldots, P_n\in S_{\mathcal P(^N X)}$ and every $\varepsilon>0$ there exists $f\in S_{X^*}$ such that, if we define $Q:=f^N\in \mathcal P(^N X)$, we have that
$$\Vert P_i+Q\Vert>2-\varepsilon$$
holds for every $i\in\{1,\ldots, n\}$.
\end{proposition}

\begin{proof}
Given $P_1,\ldots, P_n\in S_{\mathcal P(^N X)}$ and $\varepsilon>0$ we get, in view of the properties of the Aron-Berner extensions of the previous polynomials and by Theorem \ref{teosymmetricstrong}, that we can find an element $x^{***}\in S_{X^{***}}$ such that, defining $R:=(x^{***})^N\in \mathcal P(^N X^{**})$, we get that
$$\Vert \widehat{P_i}+R\Vert>2-\varepsilon$$
holds for every $i\in\{1,\ldots, n\}$. Now, given $i\in\{1,\ldots, n\}$ choose $x_i^{**}\in S_{X^{**}}$ such that $\widehat P_i(x_i^{**})+(x^{***}(x_i^{**}))^N>2-\varepsilon$. From weak-star denseness of $S_{X^*}$ in $S_{X^{***}}$ we can find $f\in S_{X^*}$ such that $\widehat{P}_i(x_i^{**})+f(x_i^{**})^N>2-\varepsilon$ holds for every $i\in\{1,\ldots, n\}$. Now, fix an arbitrary $i\in\{1,\ldots, n\}$ and find, by \cite[Theorem 1]{dg}, a net $\{x_s^i\}\in S_X$ such that $P(x_s^i)\rightarrow \widehat P(x_i^{**})$ holds for all $P\in \mathcal P(^N X)$. Pick $s$ large enough so that $P_i(x_s^i)+(f(x_s^i))^N>2-\varepsilon$. Hence, if we define $Q:=f^N\in \mathcal P(^N X)$, we get
$$\Vert P_i+Q\Vert>P_i(x_s^i)+(f(x_s^i))^N>
2-\varepsilon.$$
Since $\varepsilon>0$ was arbitrary we conclude the proposition.
\end{proof}

\begin{remark}
In contrast with what is known for the SD2P and the SSD2P, we do not know whether the sequential SD2P passes from a bidual Banach space $X^{**}$ to $X$. That is the reason why Proposition \ref{propobiduSD2P} has been proved by making use of theory of approximation of polynomials.
\end{remark}

We will finish with a further generalisation of Theorem \ref{teosymmetricstrong}, which will allow us to improve \cite[Theorem 3.2]{ab}. In order to do so, we need to introduce the following notation, coming from \cite[Section 3]{ab}. Given a Banach space, we consider the sequence of all its even duals
$$X\subseteq X^{**}\subseteq X^{(4}\subseteq \ldots \subseteq X^{(2n}\subseteq \ldots,$$ 
and consider $X^{(\infty}$ as the completion of the normed space $\bigcup\limits_{n=0}^\infty X^{(2n}$. Given $N\in\mathbb N$, the previous sequence joint with the Aron-Berner extension of a polynomial defines the following chain of isometric embeddings
$$\mathcal P(^N X)\subseteq \mathcal P(^N X^{**})\subseteq \ldots \subseteq \mathcal P(X^{2n})\subseteq\ldots,$$
which gives rise to a completion to
$$\mathcal P(^N X)\subseteq \ldots \subseteq \mathcal P(^N X^{(2n})\subseteq \ldots \subseteq \mathcal P(^N X^{(\infty}).$$
This inclusion is defined as follows: given a polynomial $P\in\mathcal P(^N X)$ and $x\in \bigcup\limits_{n=0}^\infty X^{(2n}$ we define $P(x)=P^{(2n}(x)$ if $x\in X^{(2n}$. This defines a continuous polynomial on $\bigcup\limits_{n=0}^\infty X^{(2n}$, which extends in a unique way by continuity to a polynomial $\bar P:X^{(\infty}\longrightarrow \mathbb R$, which is an element of $P(^N X^{(\infty})$. Clearly, $\Vert \bar P\Vert=\Vert P\Vert$ holds since the Aron-Berner extension is an isometric embedding.

Now we are ready to prove the following extension of Theorem \ref{teosymmetricstrong}, following the lines of \cite[Theorem 3.2]{ab}.

\begin{theorem}\label{propiPinfinito}
Let $X$ be a Banach space and let $Y:=X^{(\infty}$. Assume that, for every $P_1,\ldots, P_n\in S_{\mathcal P(^N Y)}$ and every $\varepsilon>0$ there exists an element $\varphi\in S_{Y^*}$ such that
$$\Vert P_i+\varphi^N\Vert>2-\varepsilon$$ 
holds for every $i\in\{1,\ldots, n\}$. Then the norm of $\mathcal P(^N X)$ is octahedral.
\end{theorem}

\begin{proof}
Let $P_1,\ldots, P_n\in S_{\mathcal P(^N X)}$ and $\varepsilon>0$. Given $i\in\{1,\ldots, n\}$ we denote by $\bar P_i$ its canonical extension to $S_{\mathcal P(^N Y)}$. By assumptions we can find $\varphi\in S_{Y^*}$ such that
$$\Vert \bar P_i+\varphi^N\Vert>2-\varepsilon$$
holds for every $i\in\{1,\ldots, n\}$. For every $i\in\{1,\ldots, n\}$ pick $z_i\in B_Y$ such that $\bar P_i(z_i)+\varphi(z_i)^N>2-\varepsilon$. Now, from a denseness argument we can assume with no loss of generality that $z_i\in X^{(2m}$ for certain $m\in\mathbb N$. Since $B_{(X^*)^{(\infty}}$ is weak-star dense in $B_{Y^*}$ by \cite[Proposition 3.1]{ab} we can find $k\geq m$ and $\phi\in S_{X^{(2k+1}}$ such that $\bar P_i(z_i)+\phi(z_i)^N>2-\varepsilon$ holds for every $i\in\{1,\ldots, n\}$. If we consider $z_i\in X^{(2m}\subseteq X^{(2k}$ and taking into account the definition of $\bar P$ we get that
$$\Vert P_i^{(2k}+\phi^N\Vert\geq P_i^{(2k}(z_i)+\phi(z_i)^N>2-\varepsilon$$
holds for every $i\in\{1,\ldots, n\}$. Now, an inductive argument similar to that of the proof of Proposition \ref{propobiduSD2P} yields an element $f\in B_{X^*}$ such that $\Vert P_i+f^N\Vert>2-\varepsilon$ holds for every $i\in\{1,\ldots, n\}$, and the proof is finished.\end{proof}

\section{Applications of the main result}\label{section:examples}

Let us exhibit examples of Banach spaces with the sequential SD2P to get examples of projective symmetric tensor product spaces with the SD2P. To begin with, we will generalise two known results about the diameter two properties in the projective symmetric tensor products.

\begin{example}
It is clear that ASQ Banach space enjoy the sequential SD2P, so Corollary \ref{coromejoraSD2P} generalises the main results of \cite[Section 3]{blr}.
\end{example}

\begin{example}\label{examinfidimcentra}
All the Banach spaces $X$ whose centralizer is infinite-dimensional and $\mathrm{ext}(B_X)\neq \emptyset$ have the sequential SD2P (it follows from the proof of \cite[Lemma 2.2]{abr}). Indeed, given such a Banach space $X$, $x_1,\ldots, x_n\in S_X$ and $p\in \ext{B_X}$, then it remains to follow word-by-word the proof of \cite[Lemma 2.2]{abr} and define $y_n^i:=(1-f_n)x_i$ and $z_n:=f_np$ to get the desired result.
\end{example}

From the previous example we get the following corollary.

\begin{corollary}\label{centrainfi}
Let $X$ be a Banach space and assume that $Z(X^{(\infty})$ is infinite-dimensional. Then $\widehat{\otimes}_{\pi, s,N}X$ has the SD2P for every natural number $N\in\mathbb N$.
\end{corollary}

\begin{proof}
Note that, if we denote by $Y:=X^{(\infty}$, then the fact that $Z(Y)$ is infinite-dimensional implies that $Z(Y^{**})$ is infinite-dimensional by \cite[Corollary I.3.15]{hww}. Now Theorem \ref{teosymmetricstrong} together with Proposition \ref{propobiduSD2P} imply that $X$ satisfies the assumptions of Theorem \ref{propiPinfinito}, so the norm of $\mathcal P(^N X)=(\widehat{\otimes}_{\pi,s,N}X)^*$ is octahedral or, equivalently, $\widehat{\otimes}_{\pi,s,N}X$ has the SD2P for every $N\in\mathbb N$, as desired.
\end{proof}

\begin{remark}
The previous corollary implies \cite[Theorem 3.2]{ab}, where the authors obtained that $\widehat{\otimes}_{\pi, s, N}X$ has the D2P under the same assumptions. 
\end{remark}

We pass now to exhibit two different examples of Banach spaces which enjoy the sequential SD2P, providing new examples of projective symmetric tensor product spaces enjoying the SD2P.

\begin{example}\label{examplelipschitz}
Let $M$ be a pointed metric space and $\mathrm{Lip}_0(M)$ the space of all Lipschitz functions from $M$ to $\mathbb R$ which vanish at $0$ equipped with the classical Lipschitz norm given by
$$\Vert f\Vert:=\sup\limits_{x\neq y\in M}\frac{f(x)-f(y)}{d(x,y)}.$$
Then $\mathrm{Lip}_0(M)$ has the sequential SD2P if $M'$ is infinite.
\end{example}

For the proof we will need the following lemma, which is an extension of \cite[Lemma 1.5]{ccgmr}.

\begin{lemma}\label{lematecnilipschitz}
Let $M$ be a pointed metric space and let $\{f_n\}$ be a sequence in $S_{\mathrm{Lip}_0(M)}$. Denote by $U_n:=\{x\in M:f_n(x)\neq 0\}$ for every $n\in\mathbb N$ and assume that, for every $n\in\mathbb N$, there are elements $x_n\notin \bigcup\limits_{k\in\mathbb N} U_k$ such that
$$d(y,x_n)< d(y,x)$$
holds for every $y\in U_n$ and every $x\in \bigcup\limits_{k\neq n}U_k$ (in particular the sequence $\{U_n\}$ is pairwise disjoint). Then the sequence $\{f_n\}$ is equivalent to the $c_0$ basis.
\end{lemma}

\begin{proof}Consider $\lambda_1,\ldots, \lambda_N\in\mathbb R$ and let $g:=\sum_{i=1}^N \lambda_i f_i$. Since the sequence $\{U_n\}$ is pairwise disjoint then the proof of \cite[Lemma 1.5]{ccgmr} implies that
$$\Vert g\Vert\leq 2\max\limits_{1\leq i\leq N}\vert\lambda_i\vert.$$ 
In order to get an inequality from below consider $i\in\{1,\ldots, N\}$ such that $\vert\lambda_i\vert=\max\limits_{1\leq j\leq N}\vert\lambda_j\vert$. Now, since $\Vert f_i\Vert=1$, we can find $u,v\in M, u\neq v$ such that $\frac{f_i(u)-f_i(v)}{d(u,v)}>1-\varepsilon$. Notice that either $u$ or $v$ belongs to $U_i$. Assuming that $u\in U_i$ we have two possibilities:
\begin{enumerate}
\item $f_i(v)\neq 0$. This implies that $v\in U_i$ and, from the disjointness of the sequence $\{U_n\}$, we get that $f_j(u)=f_j(v)=0$ holds for every $j\neq i$. Consequently
$$\Vert g\Vert\geq \frac{\vert g(u)-g(v)\vert}{d(u,v)}=\vert \lambda_i\vert\frac{\vert f_i(u)-f_i(v)\vert}{d(u,v)}\geq \vert\lambda_i\vert(1-\varepsilon).$$
\item $f_i(v)=0$. Thus $f_i(u)>(1-\varepsilon)d(u,v)$. Now if $v\notin\bigcup\limits_{j\neq i}U_j$ then the same argument of (1) applies to get $\Vert g\Vert\geq (1-\varepsilon)\vert\lambda_i\vert$. On the other hand, if $v\in U_j$ for some $j\neq i$ we have
$$\Vert g\Vert\geq \frac{\vert g(u)-g(x_i)\vert}{d(u,x_i)}\geq \vert \lambda_i\vert \frac{f_i(u)}{d(u,x_i)}>\vert \lambda_i\vert \frac{(1-\varepsilon)d(u,v)}{d(u,x_i)}\geq \vert \lambda_i\vert(1-\varepsilon),$$
where the last inequality follows from the assumption on $x_i$.
\end{enumerate}
In any case, from the arbitrariness of $\varepsilon$ and the choice of $\lambda_i$ we get that
$$\max\limits_{1\leq i\leq n}\vert\lambda_i\vert\leq \left\Vert \sum_{i=1}^N\lambda_i f_i\right\Vert\leq 2\max\limits_{1\leq i\leq N}\vert\lambda_i\vert,$$
and the lemma is proved.
\end{proof}

\begin{proof}[Proof of Example~\ref{examplelipschitz}]
Notice that, since $M'$ is infinite, we can inductively construct a sequence of pairwise disjoint balls $\{B(x_n,r_n)\}$ in $M$ such that $x_n\in M'$ holds for every $n\in\mathbb N$. Notice now that we can assume, up considering close enough points to the center of the balls and considering smaller radii, that there exists $z_n\in M\setminus\bigcup\limits_{n\in \mathbb N} B(x_n,r_n)$ such that
$$d(z_n,x)< d(x,y)$$
holds for every $x\in B(x_n,r_n)$ and every $y\in B(x_m,r_m)$ with $n\neq m$. Now the result follows repeating word-by-word the proof of \cite[Lemma 5.4]{ccgmr} by working with a sequence of balls $B(x_n,r_n')$, for $r_n'<r_n$ small enough, and taking into account that the sequences involved in the proof are either null or equivalent to the $c_0$ basis in our case because of Lemma~\ref{lematecnilipschitz}.
\end{proof}

The following example is a generalisation of \cite[Theorem 2.2]{nw}, but the proof will be strongly based on its original proof.

\begin{example}
If $X\subseteq \mathcal C(K)$ is an infinite-dimensional uniform algebra, i.e. a closed subalgebra of $\mathcal C(K)$ which separates the points of $K$ and contains the constant functions, then $X$ enjoys the sequential SD2P.
\end{example}

Before beginning with the proof, let us introduce some notation used in \cite{nw}. Given a uniform algebra on a compact space $K$, a point $x\in K$ is said to be a \textit{strong boundary point} if, for every neighbourhood $V$ of $x$ and every $\delta>0$, there exists $f\in S_X$ such that $f(x)=1$ and $\vert f\vert<\delta$ on $K\setminus V$. The \textit{Silov boundary} of $X$, denoted by $\partial_X$ following the notation of \cite{gamelin}, is the closure of the set of all strong boundary points. It is a fundamental result of the theory of uniform algebras that $X$ can be indentified as an uniform algebra on its Silov boundary (see \cite{nw}). This fact allows us to assume, with no loss of generality, that the Silov boundary of $X$ is dense in $K$.

 \begin{proof}
First, assume that the set of isolated points of $K$ is infinite. Then there exists a sequence of different isolated points $\{x_n\}$ in $K$ and, since the set of strong boundary points is dense in $K$, then every $x_n$ is a strong boundary point. Now pick $f_1,\ldots, f_k\in S_X$. Given $i\in\{1,\ldots, k\}$ if the sequence $\{f_i(x_n)\}_n\rightarrow 0$ define $g_i^n=0$. Otherwise, define $g_i^n(x_n)=f_i(x_n)$ and $0$ otherwise (such construction can be done because $x_n$ is an isolated strong boundary point for every $n\in\mathbb N$). Notice that, since the functions of $\{g_i^n\}_n$ have disjoint support, then $\{g_i^n\}_n$ is either norm-null or equivalent to the $c_0$ basis. Now define $h_i^n:=f_i-g_i^n$, which verifies that $\Vert h_i^n\Vert\rightarrow 1$ and $h_i^n-f_i$ is either norm-null or equivalent to the $c_0$ basis. Now, defining $\phi_n:=\chi_{\{x_n\}}$ we have that $\Vert \phi_n\Vert=1$ for every $n\in\mathbb N$, $\{\phi_n\}$ is isometric to the $c_0$ basis since the supports are pairwise disjoint and, by construction, $\Vert h_i^n\pm \phi_n\Vert\rightarrow 1$.

Now assume that the set of isolated points is finite, then $K'$ is clopen and perfect. Thus we can find a sequence of pairwise disjoint open sets $\{V_n\}\subseteq K'$, which are still open in $K$. By the denseness of the set of strong boundary points we can choose $t_n\in V_n$ being a strong boundary point for every $n\in\mathbb N$. Pick $f_1,\ldots, f_k\in S_X$. Given $i\in\{1,\ldots, k\}$, if $\{f_i(t_n)\}\rightarrow 0$ then define $g_n^i$ as the constant $0$ for every $n\in\mathbb N$ and $\delta_i=1$; otherwise assume, up taking a further subsequence, that $\vert f_i(t_n)\vert\geq \delta_i$, for some $\delta_i>0$. Define $\delta:=\min\limits_{1\leq i\leq k} \delta_i$, choose $0<\varepsilon<\frac{\delta}{2}$ and find a sequence of positive numbers $\{\varepsilon_n\}$ such that $\sum_{n=1}^\infty \varepsilon_n<\varepsilon$. Define, for every $i\in\{1,\ldots, k\}$ such that $f_i(t_n)$ is not null, by using that $t_n$ is a strong boundary point, a function $g_n^i\in X$ such that $\Vert g_n^i\Vert=\vert f_i(t_n)\vert$, $g_n^i(t_n)=f_i(t_n)$ and that $\vert g_n^i(t)\vert<\varepsilon_n$ holds for every $t\in K\setminus V_n$. Finally consider $h_n^i:=f_i-g_n^i$ for every $i\in\{1,\ldots, k\}$ and every $n\in\mathbb N$. It is not difficult to prove that $\Vert h_n^i\Vert\rightarrow 1$ for every $i\in\{1,\ldots, k\}$. Furthermore, given $i\in\{1,\ldots, k\}$, then $g_n^i=f_i-h_n^i$ is either norm-null (in the case that $f_i(t_n)\rightarrow 0$) or equivalent to the $c_0$ basis (otherwise). In fact, given $i\in\{1,\ldots, k\}$ such that $\vert f_i(t_n)\vert\geq \delta_i>0$, select an arbitrary $p\in\mathbb N$ and $\lambda_1,\ldots, \lambda_p\in \mathbb R$. Define $g:=\sum_{j=1}^p\lambda_j g_j^i$. For a given $t\in K$ we have two possibilities:
\begin{enumerate}
\item There is $q\in\{1,\ldots, k\}$ such that $t\in V_q$. Then, by the construction, $\vert g_j^i(t)\vert\leq \varepsilon_j$ for every $j\neq q$. Consequently, the triangle inequality implies that
\[\begin{split}
\vert g(t)\vert\leq \sum_{j=1}^n \vert \lambda_i\vert \vert g_j^i(t)\vert& \leq \max\limits_{1\leq j\leq p} \vert \lambda_j\vert \left( \vert g_q^i(t)\vert+\sum_{j\neq q}\varepsilon_j\right )\\ & \leq (1+\varepsilon)\max\limits_{1\leq j\leq p}\vert \lambda_j\vert. 
\end{split}
\]
\item If $t\notin\bigcup\limits_{j=1}^p V_j$ then $\vert g_j^i(t)\vert\leq \varepsilon_j$ holds for every $j\in\{1,\ldots, p\}$, and so similar estimates to the ones of the previous case yields that
$$\vert g(t)\vert\leq \max\limits_{1\leq j\leq p}\vert \lambda_j\vert\varepsilon.$$
\end{enumerate}
Hence, taking supremum in $t\in K$, we get that $\Vert \sum_{j=1}^p\lambda_j g_j^i\Vert\leq (1+\varepsilon)\max\limits_{1\leq j\leq p}\vert \lambda_j\vert$. In order to get an inequality from below choose $q\in\{1,\ldots, p\}$ such that $\vert \lambda_q\vert=\max\limits_{1\leq j\leq p} \vert\lambda_j\vert$. Then
\[\begin{split}
\vert g(t_q)\vert\geq \vert \lambda_q\vert\vert g_q^i(t_q)\vert-\sum_{j\neq q}\vert \lambda_j\vert \vert g_j^i(t_q)\vert& \geq \vert\lambda_q\vert \delta-\vert \lambda_q\vert \sum_{j\neq q} \varepsilon_j\\
& \geq\vert \lambda_q\vert(\delta-\varepsilon)\\
& \geq\max\limits_{1\leq j\leq p}\vert\lambda_j\vert \frac{\delta}{2}.
\end{split}\]
Consequently
$$\frac{\delta}{2}\max\limits_{1\leq j\leq p} \vert\lambda_j\vert\leq \left\Vert \sum_{j=1}^p\lambda_j g_j^i\right\Vert\leq (1+\varepsilon)\max\limits_{1\leq j\leq p} \vert\lambda_j\vert,$$
which proves that the sequence $\{g_j^i\}_j$ is equivalent to the $c_0$ basis.

For every $n\in\mathbb N$, we can find from the continuity of $h_n^i$ an open set $W_n$ such that $t_n\in W_n\subseteq V_n$ and such that $\vert h_n^i(t)-h_n^i(t_n)\vert<\varepsilon_n$ holds for every $t\in W_n$ and every $i\in\{1,\ldots, k\}$. Using again that $t_n$ is a strong boundary point find $\phi_n\in S_X$ such that $\phi(t_n)=1$ and $\vert \phi_n(t)\vert<\varepsilon_n$ holds for all $t\in K\setminus W_n$ and $n\in\mathbb N$. It is clear that $\{\phi_n\}$ is a sequence equivalent to the $c_0$ basis by a similar argument to that of the sequence $\{g_j^i\}_j$ given above. Let us prove that, for a given $i\in\{1,\ldots, k\}$, then $\Vert h_n^i\pm \phi_n\Vert\rightarrow 1$. Given $n\in\mathbb N$ and $t\in K$ we have two possibilities:
\begin{enumerate}
\item If $t\notin W_n$ then $\vert h_n^i(t)\pm \phi_n(t)\vert\leq \Vert h_n^i\Vert+\varepsilon_n$.
\item If $t\in W_n$ then $\vert h_n^i(t)\pm \phi_n(t)\vert\leq 1+\vert h_n^i(t_n)\vert+\vert h_n^i(t)-h_n^i(t_n)\vert<1+\varepsilon_n+\vert h_n^i(t_n)\vert$.
\end{enumerate}
So, taking supremum in $t$, we get that
$$\Vert h_n^i\pm \phi_n\Vert\leq \max\{\Vert h_n^i\Vert+\varepsilon_n, 1+\varepsilon_n+\vert h_n^i(t_n)\vert\}.$$
From the construction of the sequence $\{h_n^i\}$ it follows that the previous maximum tends to $1$, as desired.
\end{proof}

\end{document}